\documentclass{article}

\usepackage{arxiv}

\usepackage[utf8]{inputenc} 
\usepackage[T1]{fontenc}    
\usepackage{hyperref}       
\usepackage{url}            
\usepackage{booktabs}       
\usepackage{amsfonts}       
\usepackage{nicefrac}       
\usepackage{microtype}      
\usepackage{lipsum}
\usepackage{graphicx}

\usepackage{algorithm}
\usepackage{algpseudocode}
\usepackage{tikz}
\usepackage[nocomma,short]{optidef}
\usepackage{natbib}
\usepackage{amsthm}
\usepackage{amsmath}
\usepackage{amssymb}
\newtheorem{theorem}{Theorem}[section]
\newtheorem{proposition}[theorem]{Proposition}

\theoremstyle{definition}
\newtheorem{definition}[theorem]{Definition}

\theoremstyle{remark}

\title{Physics-Informed Tensor Completion with Application to Distribution System State Estimation}

\author{
 Alexander While \\
  The Ohio State University\\
  Columbus, OH, USA \\
  \texttt{while.1@osu.edu} \\
   \And
 Anil Aswani \\
  University of California Berkeley\\
  Berkeley, CA, USA \\
  \texttt{aaswani@berkeley.edu} \\
  \And
 Chen Chen \\
  The Ohio State University\\
  Columbus, OH, USA \\
  \texttt{chen.8018@osu.edu} \\
  \AND
  William N. Caballero \\
  Air Force Institute of Technology \\
   Wright-Patterson Air Force Base, OH, USA\\
  \texttt{william.caballero@afit.edu} \\
}

\begin{document}

\maketitle
\begin{abstract}
State estimation in power distribution networks often involves sparse, noisy, and heterogeneous measurements, making classical weighted least-squares methods unreliable in low-observability regimes. To meet this challenge we develop a physics-informed low-rank tensor completion method. The method encodes voltage phasor components, voltage magnitudes, and power injections over time in a low-rank tensor, while incorporating physics through convex residual penalties. For distribution network state estimation, these residuals are instantiated using linearized power-flow equations, yielding a convex optimization problem over a tensor gauge ball, equivalently represented as a polytope implicitly described via integer programming. We solve this problem using globally convergent blended conditional gradients, resulting in a general-purpose algorithmic framework that can be adapted across different domains and applications where low-rank structure and convex physics residuals are available. We also derive a deterministic error bound showing that, under a restricted stability condition, the estimation error is controlled by measurement noise and physics-model mismatch. Numerical experiments show that the proposed method substantially improves state recovery compared with purely data-driven tensor completion and weighted least-squares, especially in low-observability regimes with missing data.
\end{abstract}%






\section{Introduction}
State estimation is the task of reconstructing the full electrical operating state of a network, typically the voltage magnitude and phase angle at each node, from incomplete and noisy measurements (see, e.g., \citep{monticelli2012state}). This problem is critical to reliable operation of modern electric grids, which involves careful management of complex power flows across a network. State estimation is usually separated between transmission and distribution systems because these networks are operated by different entities and have different physical and informational characteristics. Transmission systems are typically meshed and heavily monitored, whereas distribution systems are often radial or weakly meshed and only sparsely instrumented, which gives rise to more frequent low-observability issues: measurements are sparse, noisy, heterogeneous, may be missing, and are collected at different temporal resolutions.

Classical weighted least-squares (WLS) state estimation is the standard approach to state estimation; however, it relies on sufficient measurement redundancy and accurate network models, and therefore can become unreliable or even infeasible when only a small fraction of the relevant quantities can be observed. Tensor completion provides a natural alternative in such settings, as it incorporates a tensor rank constraint that can help capture patterns among measurements to infer missing data: loads and generation typically vary smoothly over time; neighboring nodes exhibit spatial dependence; and network physics couple voltage phasors and power injections.

Recent work has therefore studied tensor-completion methods for distribution system state estimation (DSSE). Model-free tensor-completion approaches recover missing data entries or states by fitting low-rank tensor factorizations to observed measurements \citep{Liu2022stateest}. These methods are attractive because they can operate without an accurate feeder model, but the resulting estimates are not guaranteed
to satisfy physical laws. Model-based tensor-completion methods address this limitation by incorporating power-flow equations, Ohm's law, or other distribution-network constraints into the completion problem through physics-informed or physics-regularized frameworks \citep{madbhavi2020tensor,madbhavi2021stateest,Ghasemkhani2021pmu,zhu2024accelerated,yu2026noisy}. Physics-informed tensor completion has also recently emerged outside power systems, including reconstructed infrared thermography \citep{yang2026research}, and passive localization in sensor networks \citep{zhang2026reconstructing}.

A common limitation of this literature is that the modeling and algorithmic
choices are bespoke and tightly coupled to the application. Existing methods often rely on application-specific tensor structures, domain-dependent factorizations,
hand-derived algorithmic updates, or simulator-specific adjoint calculations.
Thus, they typically focus on how to incorporate one domain's
physics into one particular tensor-completion model. Another issue is reliance on Alternating Direction Method of Multipliers (ADMM) or Alternating Least Squares (ALS), which are standard options that can often deliver good solutions, but do not in general certify global optimality.

This paper addresses both gaps. We adapt a globally convergent Blended Conditional Gradients (BCG) approach to a physics-informed tensor completion model that accommodate a wide range of physics. Our approach builds on recent work on gauge-norm tensor completion
\citep{nonnegten,chen2025tensor,pan2023accelerated}, which provides a
data-efficient convex model for low-rank tensor recovery.  We extend this model by incorporating physical information through convex residual penalties. In the DSSE application considered here, these residuals are instantiated using linearized power-flow equations, but the formulation applies to a broad class of physics that can be reasonably expressed, approximated, or relaxed in terms of such penalties. 

We also provide a deterministic error analysis for our proposed state estimator. The
analysis introduces a restricted stability condition over the feasible
perturbation set induced by the tensor gauge constraint. Under this condition, we show that estimation error is controlled by two quantities: the measurement noise on observed entries and the physics-model mismatch introduced by the linearized physical residual.

Numerical experiments demonstrate the practical benefits of the proposed approach.  The results show that incorporating physics through convex residual penalties substantially improves recovery of voltage magnitudes and angles, especially in low-observability regimes with high missing-data rates. These experiments support the use of physics-informed tensor completion as a data-efficient state estimation tool for distribution networks.

\section{Models and Background}
This section describes state estimation, weighted least squares, and tensor completion.

\begin{table}
\centering
\caption{Symbology. Symbols will also be defined as they appear.}
\begin{tabular}{@{}ccl@{}}
\toprule
Symbol & Domain & Meaning                                                 \\ \midrule
$M$     & $\mathbb{Z}_{> 0}$& Number of nodes in distribution system \\
$N$     &$\mathbb{Z}_{> 0}$& Number of quantities present in state tensor \\
$T$ &$\mathbb{Z}_{> 0}$& Number of time periods observed \\
$v$ &$\mathbb{C}^{T \times M}$& Voltage phasor \\
$Re(v)$ &$\mathbb{R}^{T \times M}$& Real voltage \\
$Im(v)$ &$\mathbb{R}^{T \times M}$& Imaginary voltage \\
$|v|$ &$\mathbb{R}^{T \times M}$& Voltage magnitude (p.u.) \\
$p$ &$\mathbb{R}^{T \times M}$& Active power injection (p.u.) \\
$q$ &$\mathbb{R}^{T \times M}$& Reactive power injection (p.u.) \\
$\mathcal{Y}$   &$\mathbb{R}^{T \times N\times M}$& Measurement tensor \\
$\mathcal{W}$   &$\{0,1\}^{T \times N\times M}$& Observation mask tensor \\
$\psi$  &$\mathbb{R}^{T \times N\times M}$& State estimation tensor \\
$||\cdot||_2$ & $\mathbb{R}_{\geq 0}$ &$L^2$ norm \\
$||\cdot||_F$ & $\mathbb{R}_{\geq 0}$ & Frobenius norm \\
$||\cdot||_\pm$ & $\mathbb{R}_{\geq 0}$ & Tensor gauge norm \\
$\lambda$ &$\mathbb{R}_{>0}$& Hyperparameter \\
$n$ &$\mathbb{Z}_{>0}$& Number of observations \\
$\alpha$ & $[0,1]$& Weight parameter \\
$j$ &$\cdot$& Imaginary unit \\
$:$ & $\cdot$ & Slice of matrix or tensor along axis of index \\
$\circ$ & $\cdot$ & Hadamard product \\
$\overline{\cdot}$ & $\cdot$ & Complex conjugate \\
\toprule
\end{tabular}
\end{table}

\subsection{State Estimation and Weighted Least Squares}
We consider the problem of power system state estimation over a discrete time horizon of $T$ periods. Let the power network consist of $M$ nodes (buses), indexed by $m = 1,\dots,M$, and time periods indexed by $t = 1,\dots,T$. At each time $t$, the electrical state of the network is characterized by the complex voltage phasor vector $v^{(t)} \in \mathbb{C}^M,$ whose entries $v_m^{(t)} = |v_m^{(t)}|\exp(i\theta_m^{(t)})$ encode the voltage magnitude and phase angle at each node.

While distribution networks differ from transmission systems in topology and sensing infrastructure, the mathematical structure of the state estimation problem is essentially the same: in both settings one seeks to recover nodal voltages from partial and noisy measurements subject to physical power-flow constraints. The primary challenge in distribution systems arises from reduced observability; conversely, the size of transmission systems demand computationally scalable estimation algorithms.

The goal of state estimation is to recover the sequence $\{v^{(t)}\}_{t=1}^T$ from a collection of heterogeneous measurements. In distribution systems, these measurements are typically obtained from supervisory control and data acquisition (SCADA) devices and may include voltage magnitudes, active power injections, and reactive power injections. 

A classical approach to this problem is weighted least squares (WLS), which estimates the true underlying state by minimizing the discrepancy between observations and nonlinear measurement functions derived from the power-flow equations. Specifically, letting $h(\cdot)$ denote the mapping from the system state to measurements, WLS solves
\[
    \min_{v^{(1)},\dots,v^{(T)}} 
    \sum_{t=1}^T 
    \left\|
        y^{(t)} - h\big(v^{(t)}\big)
    \right\|_{\Sigma^{-1}}^2,
\]
where $y^{(t)}$ collects the measurements at time $t$, $\Sigma$ is a covariance matrix encoding measurement accuracy, and $\|\cdot\|_{\Sigma^{-1}}^2$ denotes a weighted squared norm. This formulation effectively assigns higher influence to more reliable measurements. 

While WLS is well understood and performs well under sufficient measurement redundancy and accurate network models, it becomes unreliable or even infeasible in low-observability regimes where measurements are sparse or missing. To address these challenges, we adopt a data-driven perspective in which the state estimation problem is reformulated as the recovery of a structured tensor from incomplete observations. 

\subsection{Tensor Completion}
\label{subsec:TC}
In data analysis, a tensor is commonly represented as a multidimensional array indexed along multiple modes, providing a natural framework for modeling multiway data. In our context, let $\psi^\star \in \mathbb{R}^{T \times N \times M}$ denote the true (but unknown) state tensor, where the second dimension indexes different physical quantities of interest. Let $\mathcal{Y} \in \mathbb{R}^{T \times N \times M}$ denote the measurement tensor, and let $\mathcal{W} \in \{0,1\}^{T \times N \times M}$ be the observation mask, where $\mathcal{W}_{t,n,m} = 1$ if the $(t,n,m)$ entry is observed and $0$ otherwise. As with WLS, we assume a standard additive noise model: $\mathcal{W} \odot \mathcal{Y} = \mathcal{W} \odot (\psi^\star + \varepsilon),$ where $\varepsilon$ represents measurement noise and $\odot$ denotes the Hadamard product. 

The state estimation problem can be viewed as an inverse problem in which one seeks to estimate $\psi^\star$ from partial observations $\mathcal{W} \odot \mathcal{Y}$ under structural and physical constraints. This is a special type of tensor completion problem, which is typically formulated as a problem of recovering an underlying tensor, given partial measurements and imposing a low-rank assumption: 

\begin{mini}
    {\psi}{\frac{1}{n}||\mathcal{W}\odot(\psi - \mathcal{Y})||^2}{\label{eq:tc}}{}
    \addConstraint{\mbox{rank}(\psi)}{\leq}{k}
\end{mini}

This is a powerful and general-purpose model, but even calculating the rank of a given 3-tensor is NP-hard. We adopt a gauge-norm approach from the literature (described in Section~\ref{subsec:BCG}) to work around this issue, which comes with a practical BCG algorithm with strong theoretical performance guarantees.  Another issue is that, while the standard rank constraint can capture a variety of hidden data patterns, it is inadequate for recovering physically meaningful solutions that satisfy nonlinear flow relationships, as shown in our experiments (see Section~\ref{sec:exp}).  In Section~\ref{sec:prop} we propose some strategies for improving physics modeling while preserving the gauge-norm approach.

\section{Proposed Methodology}
\label{sec:prop}
\noindent
In this section we develop the BCG algorithm along with three novel adaptations in order to better model physics.

\subsection{Measurements}
While the models we will describe are agnostic to the choice of measurements and physics-based penalties, we will first describe our specific choice of measurements for this study as they will be consistent throughout. We describe our choice of physics penalty in Section \ref{sec:penalized}.
Our state tensor is $\psi \in \mathbb{R}^{T \times 5 \times M}$. We arrange the real and imaginary voltage components of the voltage phasors $v$, the voltage magnitude 
$|v|$, the active power $p$, and the reactive power $q$ as follows:
\begin{align*}
    \psi_{(:,1,:)} &:= Re(v) \\
    \psi_{(:,2,:)} &:= Im(v) \\
    \psi_{(:,3,:)} &:= |v| \\
    \psi_{(:,4,:)} &:= p \\
    \psi_{(:,5,:)} &:= q.
\end{align*}
All measurements are per unit (p.u.), meaning that the voltages and powers are divided by a nominal voltage and power, respectively, that is specified within the network description. 

We also define $\psi^{SCADA}\in \mathbb{R}^{T \times 3 \times M}$ as the subtensor of $\psi$ whose elements are the measurements acquired through SCADA devices. Specifically,
\begin{align*}
    \psi^{SCADA}_{(:,1,:)} &:= |v| \\
    \psi^{SCADA}_{(:,2,:)} &:= p \\
    \psi^{SCADA}_{(:,3,:)} &:= q.
\end{align*}

\subsection{BCG}
\label{subsec:BCG}
In \cite{chen2025tensor}, they replace the rank constraint with gauge rank, denoted $||\cdot ||_\pm$ . Let $S_\lambda$ be the set of rank-1 tensors with entries equal to -1 or 1. Let $C_\lambda = conv(S_\lambda)$. In \cite{chen2025tensor} it is shown that $S_\lambda = \lambda S_1$ and $C_\lambda = \lambda C_1$. Then for a tensor $\psi$, we define the gauge norm as 
\begin{align}
    ||\psi||_\pm := \inf\{\lambda \geq 0 : \psi\in\lambda C_1\}
\end{align}
It is established in \cite{chen2025tensor} that $||\psi||_{\text{max}}\leq||\psi||_\pm\leq\text{rank}(\psi)\cdot||\psi||_{\text{max}}$ under a regularity condition. 

Given a tensor $\mathcal{Y}$ of $n$ observations and a tensor mask $\mathcal{W}$, where each entry is 0 if that measurement is not observed and 1 otherwise, we then define the tensor completion problem as
\begin{mini}
    {\psi}{\frac{1}{n}||\mathcal{W}\odot(\psi - \mathcal{Y})||^2}{\label{eq:tcvip}}{}
    \addConstraint{||\psi||_\pm}{\leq}{\lambda}
\end{mini}
This model achieves the information-theoretic error rate for tensor completion. Though this model (\ref{eq:tcvip}) is a convex optimization problem, it is in fact NP-hard \citep{chen2025tensor}. Hence, it is not \emph{a priori} obvious how to solve this optimization problem. One of the innovations by \citep{chen2025tensor} is the design of a numerical algorithm that is guaranteed to return a global optimal solution in a polynomial number of calls to a weak-separation oracle. 

By construction, the constraint $||\psi||_\pm\leq\lambda$ is equivalent to  $\psi \in C_\lambda$. Given the NP-hardness of the tensor completion problem (\ref{eq:tcvip}), it is not possible to represent this constraint using a polynomial number of linear inequalities or convex semidefinite constraints (unless P = NP). Instead, one of the contributions of \citep{chen2025tensor} is the construction of a representation of the \emph{vertices} of $C_\lambda$ (which is the same as the set $S_\lambda$) using a polynomial number of integer-linear constraints. This construction allows the solution of separation problems over the polytope $C_\lambda$ by solving mixed-integer linear programs (MILPs), since there is always an optimal solution of a linear program over a compact polytope at one of the vertices of the polytope. The algorithm for tensor completion uses the Blended Conditional Gradients algorithm for convex optimization \citep{braun2019blended},
in a polynomial number of calls to a weak-separation oracle. Because only a weak-separation oracle is needed, the numerical algorithm for tensor completion can be sped up in one of two ways. Much of the time, solution of the MILP can be replaced with an alternating maximization heuristic that if successful provides weak separation, otherwise if the heuristic is unsuccessful then the MILP is solved. Additionally, the MILP does not need to be solved to optimality and can often be early-terminated if a sufficiently-good solution is found.  So though in the worst case, this weak-separation oracle involves solving an integer program which is also NP-Hard, use of the alternating maximization heuristic and early termination significantly reduces how many full integer program solutions are required. For more information and the complete algorithm, we refer the reader to \cite{chen2025tensor}.

\subsection{BCG into WLS}
One of the challenges of applying any tensor completion model to the state estimation problem is that one of the desired quantities, voltage angle, is not typically measured with SCADA. Therefore, entire slices of the tensor go completely unobserved and thus are difficult to complete with any accuracy. The most straightforward way to remedy this is to apply the tensor completion, which in our case is \emph{BCG}, to $\mathcal{Y}^{SCADA}$ and use the completed $\psi^{SCADA}$ measurements as input into WLS. We refer to this approach as \emph{BCG into WLS}. This two-step process may introduce compounding errors between the models and does not leverage any time dependency in the physics-based portion of the model.

\subsection{Penalized BCG}
\label{sec:penalized}
Here we propose an integrated, physics-informed approach in contrast to \emph{BCG into WLS}: we introduce a physics-informed tensor completion model by adding a physics-based penalty term to the objective of the tensor completion model. Our model is 
\begin{mini}
    {\psi}{\alpha R(\psi) + (1-\alpha)P(\psi)}{}{}
    \addConstraint{||\psi||_\pm}{\leq \lambda}
\end{mini}
where $R(\psi) := \frac{1}{n}||\mathcal{W}\odot(\psi - \mathcal{Y})||_F^2$ is the tensor residual penalty and $P(\psi)$ is a convex physics-based penalty. Note that convexity is key to maintaining favorable BCG convergence guarantees. We also introduce the weight factor $\alpha \in [0,1]$ to balance the focus of the model between the residuals and physics. When $\alpha$ is chosen \emph{a priori} and remains fixed, we refer to this as the \emph{Static Penalty} model. 

For the particular physics penalty we have chosen, we take the admittance matrix $Y_{bus}\in \mathbb{C}^{(M+1) \times (M+1)}$ and separate it into submatrices between the slack and nonslack nodes:
\begin{equation*}
    Y_{bus} = 
    \begin{bmatrix}
        Y_{00} \in \mathbb{C}^{1 \times 1} & Y_{0L} \in \mathbb{C}^{1 \times M}\\
        Y_{L0} \in \mathbb{C}^{M \times 1} & Y_{LL} \in \mathbb{C}^{M \times M}
    \end{bmatrix}.
\end{equation*}
Let $v_0$ be the voltage at the slack node. Then $w := -Y_{LL}^{-1}Y_{L0}v_0$ is the no-load voltage of the network, and we define the following coefficient matrices:
\begin{align*}
    B_1&:=Re\left(Y_{LL}^{-1}\text{diag}(\overline{w})^{-1}\right) \\
    B_2&:=Re\left(-jY_{LL}^{-1}\text{diag}(\overline{w})^{-1}\right) \\
    B_3&:=Im\left(Y_{LL}^{-1}\text{diag}(\overline{w})^{-1}\right) \\
    B_4&:=Im\left(-jY_{LL}^{-1}\text{diag}(\overline{w})^{-1}\right) \\
    C_1&:=\text{diag}\left(|w| \right)^{-1}Re\left(\text{diag}(\overline{w})Y_{LL}^{-1}\text{diag}(\overline{w})^{-1}\right) \\
    C_2&:=\text{diag}\left(|w| \right)^{-1}Re\left(-j\text{diag}(\overline{w})Y_{LL}^{-1}\text{diag}(\overline{w})^{-1}\right),  
\end{align*}
where $j$ is the imaginary unit and $\overline{\cdot}$ denotes the complex conjugate. Finally, we define the following linear power flow penalties at time $t$:
\begin{align*}
    \phi_{(1,t)}&:= B_1 \psi_{(t,4,:)}+B_2\psi_{(t,5,:)}-\psi_{(t,1,:)}+Re(w) \\
    \phi_{(2,t)}&:= B_3 \psi_{(t,4,:)}+B_4\psi_{(t,5,:)}-\psi_{(t,2,:)}+Im(w) \\
    \phi_{(3,t)}&:= C_1 \psi_{(t,4,:)}+C_2\psi_{(t,5,:)}-\psi_{(t,3,:)}+|w|.
\end{align*}

These functions relate the active and reactive powers to the different aspects of the voltage phasor: $\phi_{(1,:)}$ enforces consistency with real voltage, $\phi_{(2,:)}$ with imaginary voltage, and $\phi_{(3,:)}$ with voltage magnitude. They are based on a single iteration of a fixed point power flow calculation as done in \citep{bernstein2018load} and \citep{zhu2024accelerated}. After taking the mean of the Frobenius norm of each function, our physics-based penalty is 
\begin{equation*}
    P(\psi):=\frac{1}{3MT}\sum_{t=1}^T\sum_{i=1}^3||\phi_{(i,t)}||_F^2,
\end{equation*}
which results in the following physics-informed convex tensor completion problem:
\begin{mini}
    {\psi}{\frac{\alpha}{n}||\mathcal{W}\circ(\psi - \mathcal{Y})||_F^2+\frac{1-\alpha}{3MT}\sum_{t=1}^T\sum_{i=1}^3||\phi_{(i,t)}||_F^2}{}{}
    \addConstraint{||\psi||_\pm}{\leq \lambda}
    \label{eq:phystc}
\end{mini}
We note that each linearized residual $\phi$ is affine in the tensor variables, and so the squared residual penalty $ P(\psi)$ is convex.

\subsection{Dynamically Penalized BCG}
In the previous subsection we treated $\alpha$ as an exogenous parameter, but it is unclear how to determine a good setting \emph{a priori}. Thus we propose a \emph{Dynamic Penalty} approach, described in Algorithm \ref{alg:pitc}.  The update rule is $\alpha_{t+1} \gets \frac{R(\psi_{t+1})}{R(\psi_{t+1}) + P(\psi_{t+1})}$, with the aim of weighing the penalties in proportion to their errors in order to increase weight on the dominant source of error. We observe in our computational experiments (see Section~\ref{subsec:resdisc}) that $\alpha$ stabilizes at an effective value within the first quarter of BCG iterations.

\begin{algorithm}[h]
    \caption{Physics-Informed Tensor  with Dynamic $\alpha$} \label{alg:pitc}
    \begin{algorithmic}[1]
        \Require Tensor residual error function $R$, convex power flow error function $P$
        \State $\alpha_0 \gets 0.5$
        \State $f_0 \gets \alpha_0 R + (1-\alpha_0)P$
        \For{$t=0,\ldots ,T-1$}
            \State $\psi_{t+1}, \text{oracle} \gets \text{BCGstep}(f_t,\psi_t)$ \Comment{BCGstep is one iteration of BCG}
            \If{$\text{oracle} = \text{LPsep}$} \Comment{oracle will either be SiDO or LPsep}
                \State $\alpha_{t+1} \gets \frac{R(\psi_{t+1})}{R(\psi_{t+1}) + P(\psi_{t+1})}$
            \Else
                \State $\alpha_{t+1} \gets \alpha_t$
            \EndIf 
            \State $f_{t+1} \gets \alpha_{t+1} R + (1-\alpha_{t+1})P$
            \State $t \gets t+1$
        \EndFor
    \end{algorithmic}
\end{algorithm}

\section{Theoretical Analysis}
Here we provide an error bound for the proposed physics-informed tensor-completion estimator.

Let $\psi^\star \in \mathbb{R}^{T \times N \times M}$ denote the true
state tensor. We assume that
$
     \mathcal{W} \odot \mathcal{Y} = \mathcal{W} \odot (\psi^\star + \varepsilon),
$
 where $\varepsilon$ denotes measurement noise. 

Now, let us stack physics-based residuals in a vector  \[
    \Phi(\psi)
    :=
    \begin{bmatrix}
        \phi_{(1,1)}(\psi) \\
        \phi_{(1,2)}(\psi) \\
        \phi_{(1,3)}(\psi) \\
        \vdots \\        
        \phi_{(3,T)}(\psi)
    \end{bmatrix}
    \in \mathbb{R}^{3MT}.
\]

We shall allow for model mismatch by defining
$\eta_\Phi := \Phi(\psi^\star)$, which captures linearization error and other deviations between the true nonlinear power-flow physics and the linearized residual model. 

For notational convenience we will combine the errors as follows:

$$
    \mathcal{A}_\alpha(\psi)
    :=
    \begin{bmatrix}
        \mbox{vec}(\sqrt{\alpha/n}\, W \odot \psi) \\
        \sqrt{(1-\alpha)/3MT} \Phi(\psi)
    \end{bmatrix},
$$
where $\mbox{vec}$ denotes vectorization: stacking the elements in a vector (with order left unspecified WLOG). Thus the objective in (\ref{eq:phystc}) can be written as
$$
    \|\mathcal{A}_\alpha(\psi) - b\|_F^2,
$$
where
$$
    b
    :=
    \begin{bmatrix}
        \mbox{vec}(\sqrt{\alpha/n}\, \mathcal{W} \odot \mathcal{Y} )\\
        0
    \end{bmatrix}.
$$

Now, for a given $\lambda$ we define the feasible perturbation set around the true tensor $\psi^*$ as
$$
    \mathcal{D}(\psi^*,\lambda)
    :=
    \left\{
        \Delta : ||\psi^* + \Delta||_{\pm} \leq \lambda        
    \right\}.
$$

\begin{definition}
\(\mathcal{A}_\alpha\) satisfies \emph{restricted stability}
over \(\mathcal{D}(\psi^\star,\lambda)\) if there exists a constant
\(\kappa_\alpha > 0\) such that
$$
    \left\|
        \mathcal{A}_\alpha(\psi^\star+\Delta)
        -
        \mathcal{A}_\alpha(\psi^\star)
    \right\|_2
    \geq
    \kappa_\alpha \|\Delta\|_F
    \ \ 
    \forall \Delta \in \mathcal{D}(\psi^\star,\lambda).
   $$ 
\end{definition}

\begin{proposition}
Let $\widehat \psi$ be an optimal solution to (\ref{eq:phystc}). Suppose that $ \|\psi^\star\|_{\pm} \leq \lambda$, and moreover assume the restricted stability condition holds.  Then
\begin{align*}
    \|\widehat{\psi} - \psi^\star\|_F 
     \leq 
    \frac{2}{\kappa_\alpha}
    \left[
        \frac{\alpha}{n}
        \|\mathcal{W} \odot \varepsilon\|_F^2 
         +
        \frac{1-\alpha}{3MT}
        \|\eta_\Phi\|_F^2
    \right]^{1/2}.
\end{align*}
\end{proposition}

\begin{proof}{Proof.}
Let $\Delta := \widehat{\psi} - \psi^\star$, and so $\Delta \in \mathcal{D}(\psi^*,\lambda)$. Moreover, observe that

\begin{align*}
\mathcal A_\alpha(\widehat{\psi})-\mathcal A_\alpha(\psi^\star)
=&
\left(\mathcal A_\alpha(\widehat{\psi})-b\right)
-
\left(\mathcal A_\alpha(\psi^\star)-b\right),
\end{align*}

and so by triangle inequality,
\begin{align}
||\mathcal A_\alpha(\widehat{\psi})-\mathcal A_\alpha(\psi^\star)
||_2
\leq&
||
\mathcal A_\alpha(\widehat{\psi})-b
||_2
 +
\|
\mathcal A_\alpha(\psi^\star)-b
\|_2.
    \label{eq:trineqbd}
\end{align}

Since $\widehat{\psi}$ minimizes the objective over the feasible set and $\psi^\star$ is also feasible,
$$
    \|\mathcal{A}_\alpha(\widehat{\psi}) - b\|_2^2
    \leq
    \|\mathcal{A}_\alpha(\psi^\star) - b\|_2^2.
$$

Together with~(\ref{eq:trineqbd}) we have
$$
   ||\mathcal A_\alpha(\widehat{\psi})-\mathcal A_\alpha(\psi^\star)
||_2
    \leq
    2
    \|\mathcal{A}_\alpha(\psi^\star) - b\|_2.
$$

It remains to compute the residual of the true tensor $\psi^*$. By definition,
$$
    \mathcal{A}_\alpha(\psi^\star) - b
    =
    \begin{bmatrix}
        \sqrt{\alpha/n}\, \mbox{vec}(\mathcal{W} \odot (\psi^\star - \mathcal{Y})) \\
        \sqrt{(1-\alpha)/3MT}\, \Phi(\psi^\star)
    \end{bmatrix}.
$$
Using
$$
    \mathcal{Y} = \mathcal{W} \odot (\psi^\star + \varepsilon)
$$
on the observed entries, we have
$$
    \mathcal{W} \odot (\psi^\star - \mathcal{Y})
    =
    - \mathcal{W} \odot \varepsilon,
$$
and along with the definition of $\eta_\Phi$ we have
$$
    \|\mathcal{A}_\alpha(\psi^\star) - b\|_2^2
    =
    \frac{\alpha}{n}
    \|\mathcal{W} \odot \varepsilon\|_F^2
    +
    \frac{1-\alpha}{3MT}
    \|\eta_\Phi\|_F^2. 
$$

Applying this result to  the restricted stability condition 
yields

$$
    \kappa_\alpha \|\Delta\|_F
    \leq
    2
    \left[
        \frac{\alpha}{n}
        \|\mathcal{W} \odot \varepsilon\|_F^2
        +
        \frac{1-\alpha}{3MT}
        \|\eta_\Phi\|_F^2
    \right]^{1/2}.
$$
\end{proof}

\subsection{Discussion}
\label{subsec:disc}
Intuitively, we may interpret the bound as: if the available measurements and linearized power-flow equations are sufficiently informative over our tensor model, then small sensor noise and small physics-model error imply a small state-estimation error. 

In particular, the restricted stability condition is an observability requirement that no low-norm perturbation can be simultaneously hidden from the observed measurement mask as well as from the linearized power-flow residuals. This condition is analogous to the conditional stability assumptions used in physics-informed inverse problems, where stability of the problem converts small data and physics residuals into reconstruction-error bounds \citep{mishra2022estimates}. 

In radial distribution systems, AC power-flow structure often gives a well-conditioned, nearly one-to-one relationship between injections, branch flows, and voltages over normal operating regimes; in such cases the condition is a reasonable assumption. The condition may fail if the measurement mask is too sparse, if the linearized physics model has a large nullspace over the feasible set, or if $\lambda$ is chosen so large that the feasible perturbation set contains essentially unrestricted directions. 

 We note also that the convergence guarantees in this section apply only to the Static Penalty setting: the dynamic update is an adaptive heuristic used to select an effective penalty weight during BCG iterations. However, provided the dynamic penalty stabilizes (as observed in our experiments), the method behaves similarly to the static penalty, hence the bound can still be used for guidance.

\section{Computational Experiments}
\label{sec:exp}
Here we present our experimental results on three test cases. 

\subsection{Experimental Setup}
Experiments are conducted on three different network topologies over a one day period. Each node is equipped with a SCADA device that provides noisy measurements of voltage magnitude and active and reactive power injection. We refer to the proportion of possible SCADA measurements taken (across nodal locations and time) as the Fraction of Available Data, or FAD. In our experiments, FAD ranges between 10\% to 100\%. Note that even at 100\% FAD, the observation tensor $\mathcal{Y}$ would still only contain 60\% of the possible entries as SCADA devices do not measure complex voltages. 

The networks are based on standard IEEE test cases (9-node, 14-node, and 33-node) with data drawn from the \texttt{pandapower} 3.4.0 Python library \citep{pandapower.2018}. For each system, we simulate load data using a noisy sinusoidal model across five days at fifteen minute intervals. For each day, the minimum and maximum loads are set to 50\% and 150\% respectively of the default load values for the network within \texttt{pandapower}. At each fifteen minute interval, 1\% Gaussian measurement noise is also applied.  We randomly select four days as our training set for the $\lambda$ hyperparameter, denoting the complete tensor of all training data as $\psi_{train}$. This is done via grid search on $\lambda \in [1.0,2.0]\cdot \max(|\psi_{train}|)$ with 100\% FAD.  For experiments with reduced FAD, we uniformly sampled measurements across all nodes and time periods then applied 1\% Gaussian measurement noise to each sample. The WLS benchmark was performed using the built-in state estimation tools in \texttt{pandapower} while our tensor completion methods were written in Python 3 with Gurobi \citep{gurobi} as the integer programming solver for the BCG separation oracle. All experiments were run on a Windows laptop with 16GB of RAM and an Intel i7-12800H CPU. 

To assess the performance of each method, we calculate the mean absolute percent error (MAPE) of the estimated voltage magnitude, the mean absolute angle (MAE) of the estimated voltage angle, and the normalized mean squared error (NMSE) of the overall tensor across all nodes and time periods. We compare the performance of our model with decreasing FAD levels and compare to the WLS solution with 100\% FAD. Note that even with 90\% FAD, WLS could not always recover unobserved missing data. 

\subsection{Results and Discussion}
\label{subsec:resdisc}
The results can be seen in Figures \ref{fig-mag}, \ref{fig-ang}, and \ref{fig-nmse}. The x-axis represents FAD, with left-most representing only 10\% of nodes across time available (i.e. 90\% of node-time measurements removed with uniform probability), and right-most representing all SCADA data available; y-axes represent various types of normalized errors discussed in the setup. Since WLS failed in at least one time period in cases with less than 100\% SCADA, we provide a light dashed line for comparison to other methods that have less data available.

\begin{figure}
    \centering
    \includegraphics[width=\textwidth]{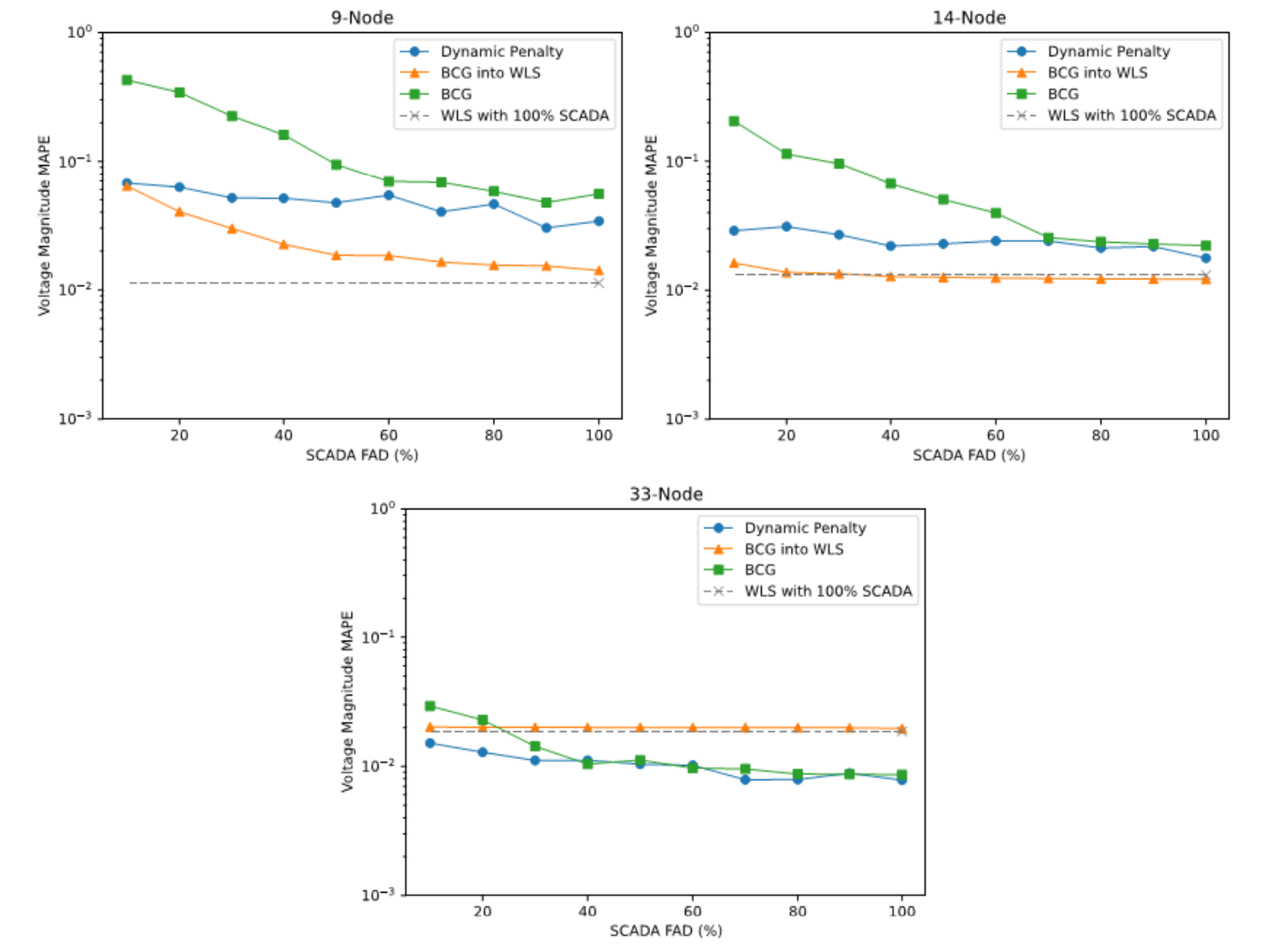}
    \caption{Voltage Magnitude Experimental Results}
    \label{fig-mag}
\end{figure}

\begin{figure}
    \includegraphics[width=\textwidth]{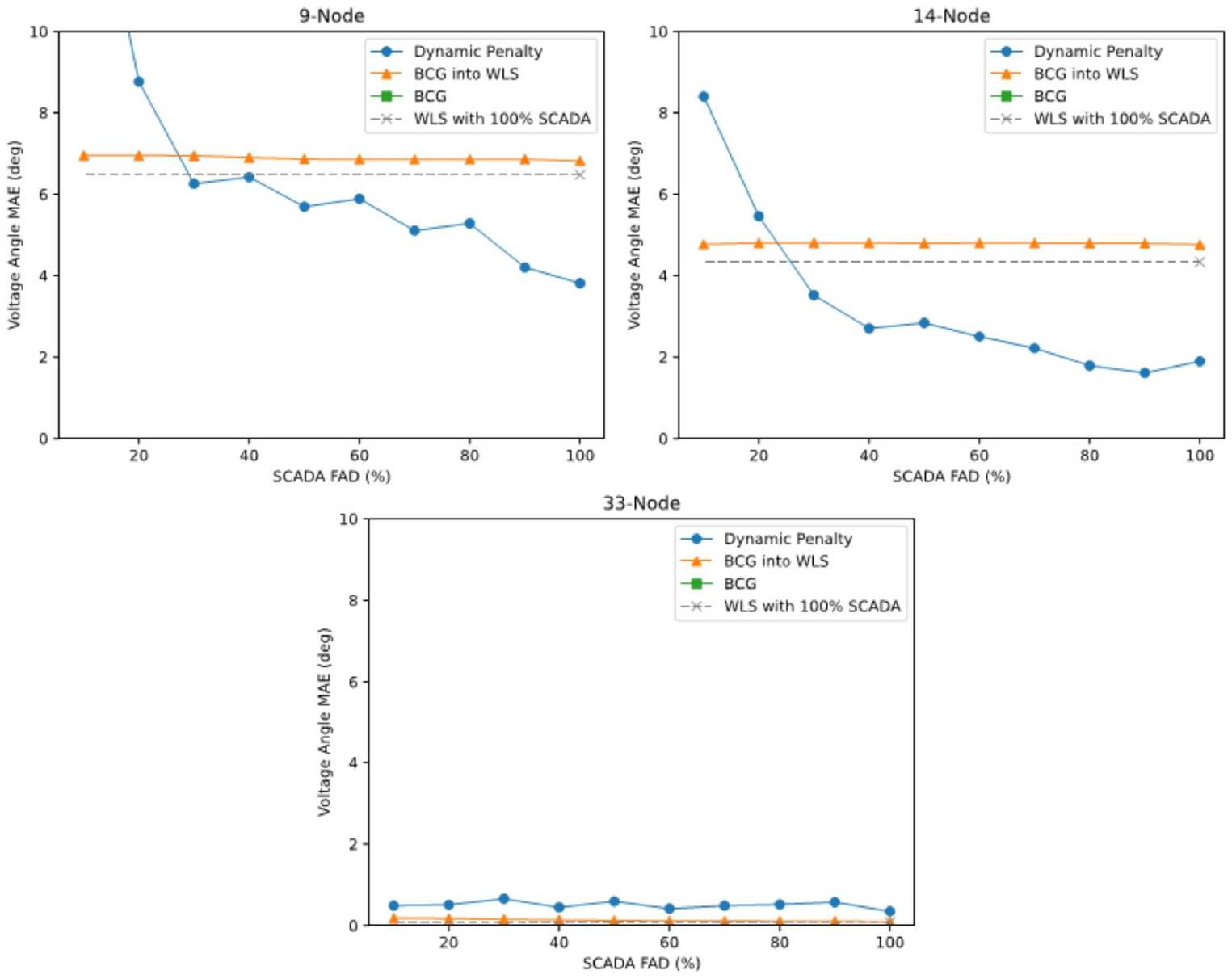}
    \caption{Voltage Angle Experimental Results. BCG had errors above $45^\circ$ in all cases and thus is not visible on the graphs.} 
    \label{fig-ang}
\end{figure}

\begin{figure}
    \includegraphics[width=\textwidth]{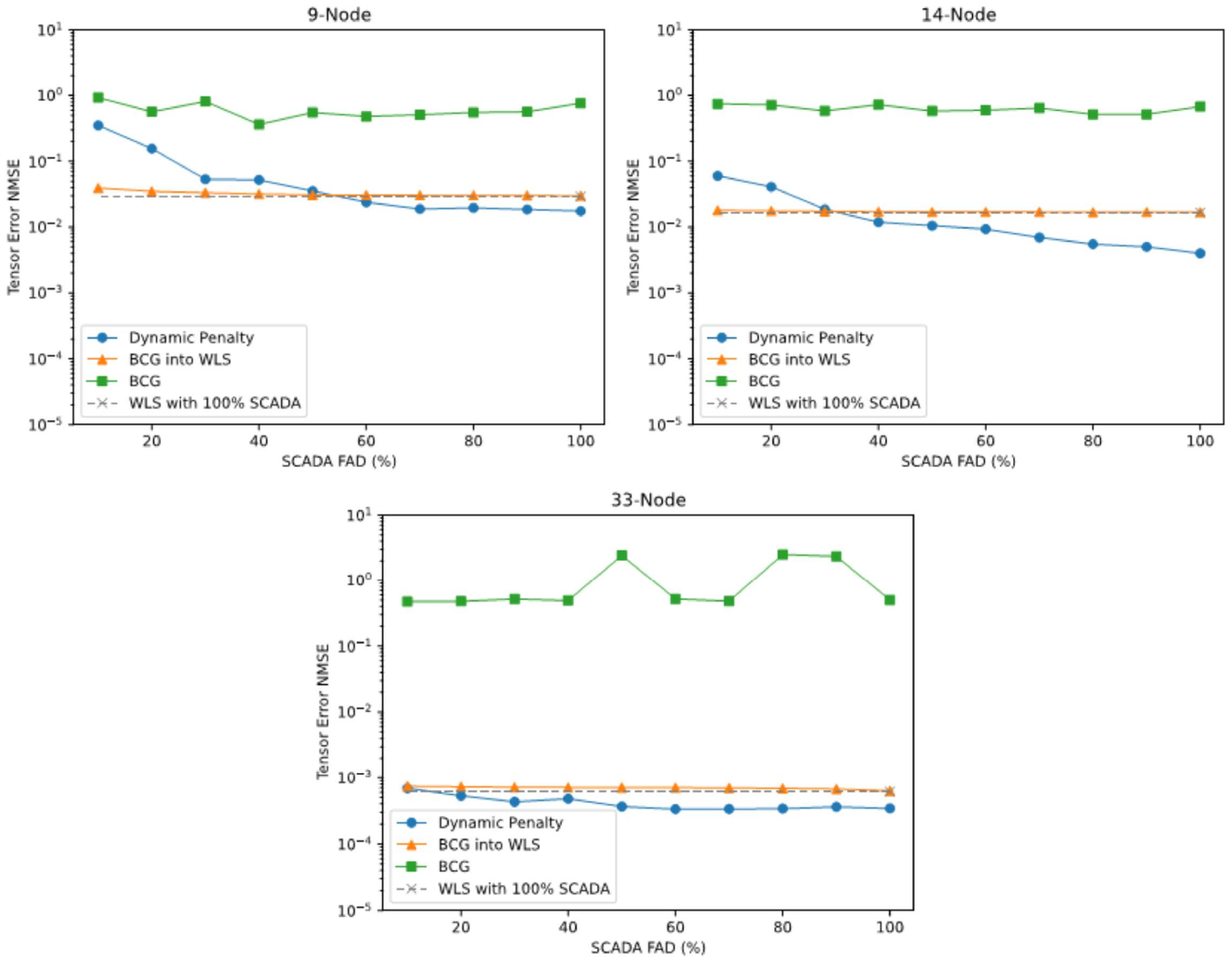}
    \caption{Tensor Error Experimental Results} 
    \label{fig-nmse}
\end{figure}

Unlike WLS, the standard (unpenalized) \emph{BCG} model can provide solutions in all cases. Our BCG-based variants, however, demonstrate the advantage of physics-based approaches. The results are not surprising: we expect \emph{BCG} to struggle as two entire slices of the tensor, the real and imaginary voltages, are unobserved. This is especially visible in the voltage angle graphs, because angle is calculated from those slices and the errors of the \emph{BCG} are so far off they are not visible on the chosen scale. The best performance of \emph{BCG} was with regards to the voltage magnitude of the 33-node network, nearly matching the best-performing dynamically chosen penalty. This is because the true state of the 33-node system contains very small voltage angles of less than half a degree, thus the real voltage is near the voltage magnitude and the imaginary voltage is near 0; therefore, the observations contained within the measurement tensor are able to well approximate the state of the system on their own. However, the voltage angle and tensor errors are still large as the real and imaginary voltages are the unobserved slices and left uncorrelated to the rest of the tensor.
In the 9- and 14- node networks we can see angles up to $25^\circ$, making the weakness of the model more pronounced. 

\emph{BCG into WLS} performs robustly across all instances. Indeed, given as little as 10\% availability in the 14- and 33-node networks and 50\% in the 9-Node network,  performance is remarkably close to standalone WLS with 100\% of data.

Dynamic Penalty has best performance up to some high level of data loss with respect to overall tensor error in all cases. Indeed, in terms of tensor error, our physics-informed approach is able to produce better estimates with 50\% of data than WLS with 100\% of data. The physics-informed tensor approach captures both intertemporal dynamics as well as power flow physics in a way that can result in better estimates.

We note a tradeoff between voltage magnitude and voltage angle errors: Dynamic Penalty has the lowest angle error in the 9- and 14- node networks even with 30\% SCADA FAD, but lags in voltage magnitude compared  to \emph{BCG into WLS}. On the 33-node network, it is reversed, with lower magnitude errors even with 10\% of data but greater angle errors. We hypothesize this correlates to the aforementioned size of the voltage angles within the system, as the 9- and 14-node networks have greater spread of angles overall with the 33-node having much smaller. Indeed, all methods demonstrate very low angle errors on the 33-node network. It should also be noted that increasing the available data does not appear to monotonically decrease the errors. We hypothesize this is affected by the distribution of the available measurements captured within the sample, as the magnitudes of the measurements and their locations within the penalties lead to greater implicit weight assigned to them. 

\subsubsection{Dynamic vs Static Penalties}
\begin{figure}
    \includegraphics[width=\textwidth]{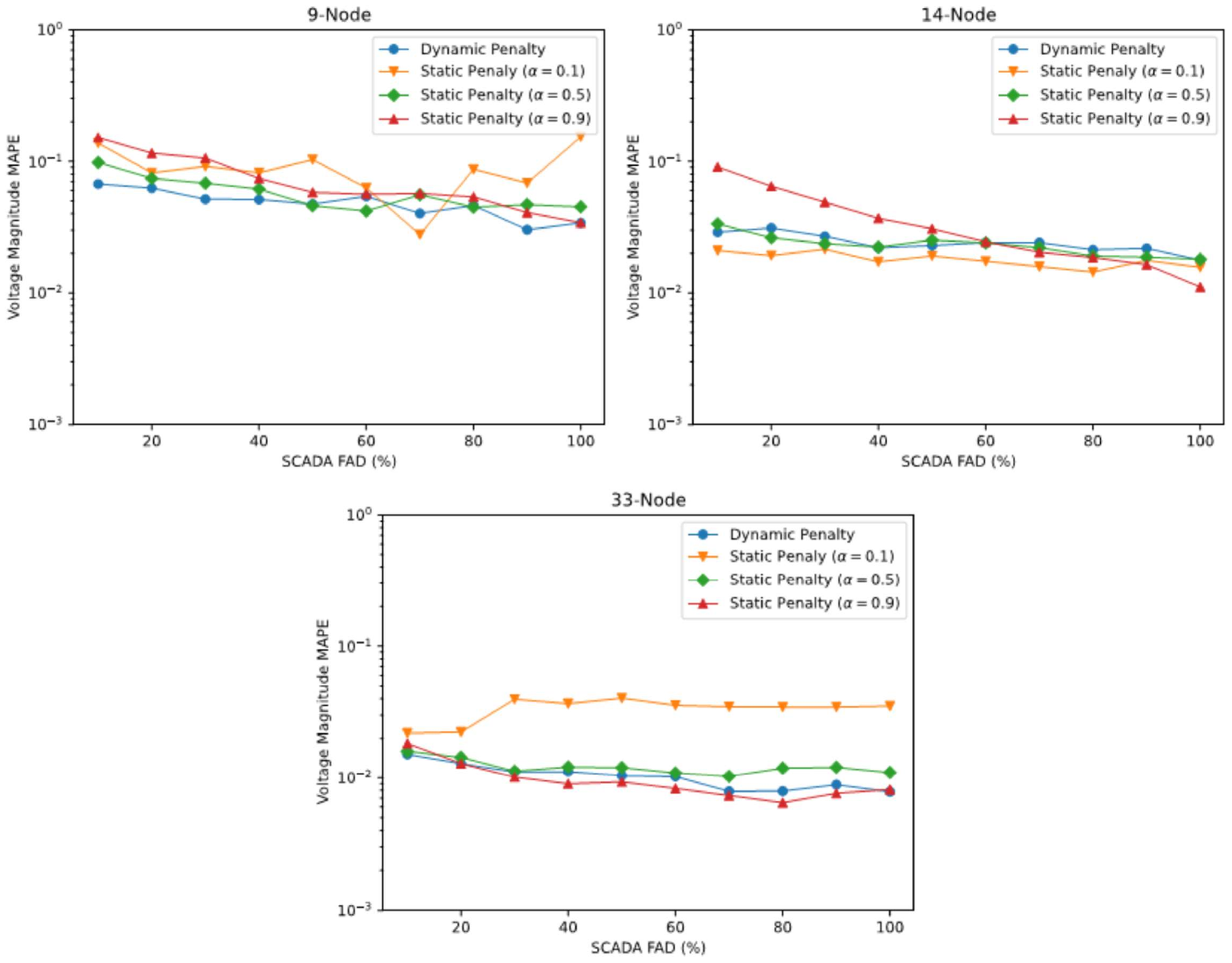}
    \caption{Voltage magnitude of static vs dynamic penalty} 
    \label{fig-static-mag}
\end{figure}

\begin{figure}
    \includegraphics[width=\textwidth]{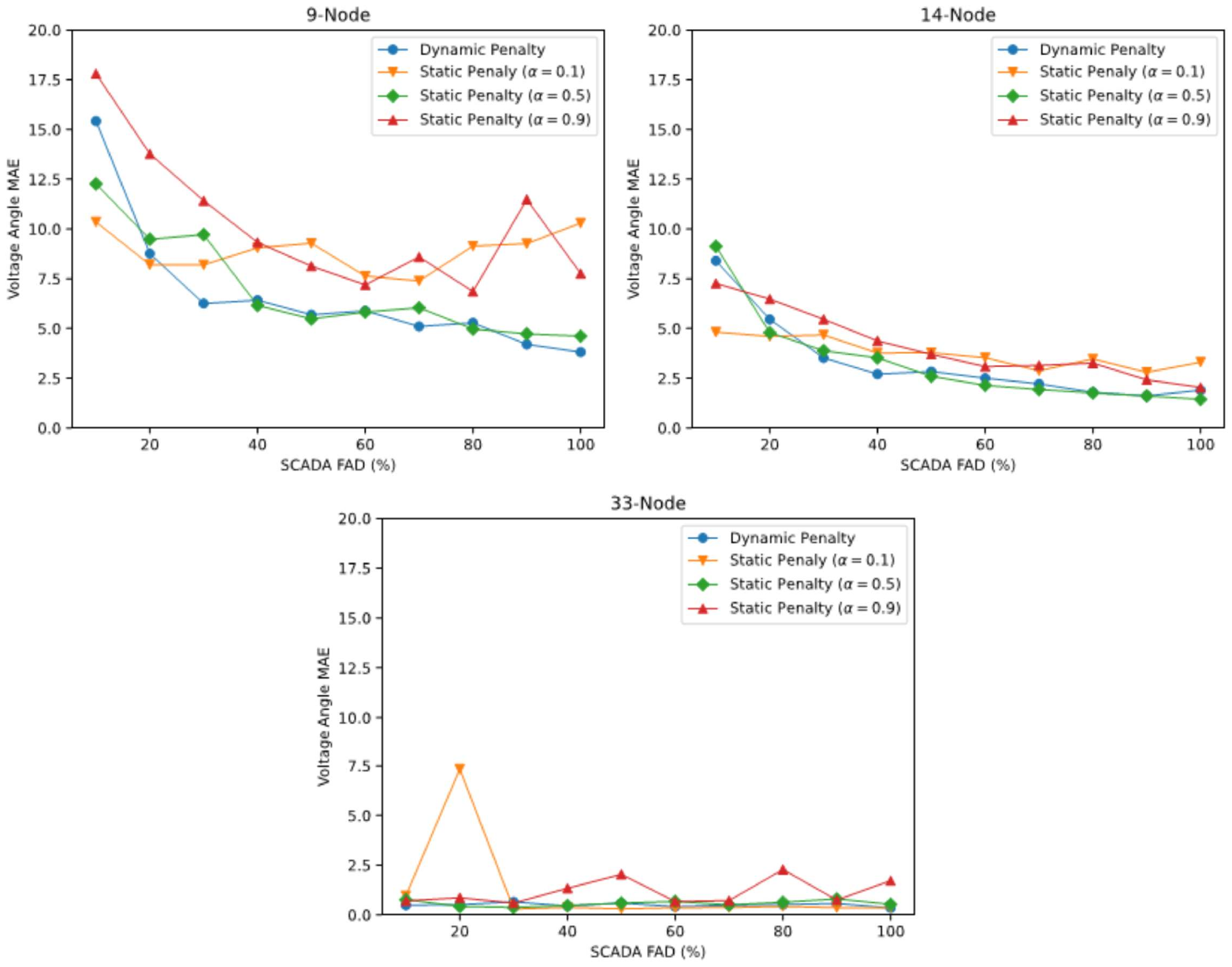}
    \caption{Voltage angle of static vs dynamic penalty} 
    \label{fig-static-ang}
\end{figure}

\begin{figure}
    \includegraphics[width=\textwidth]{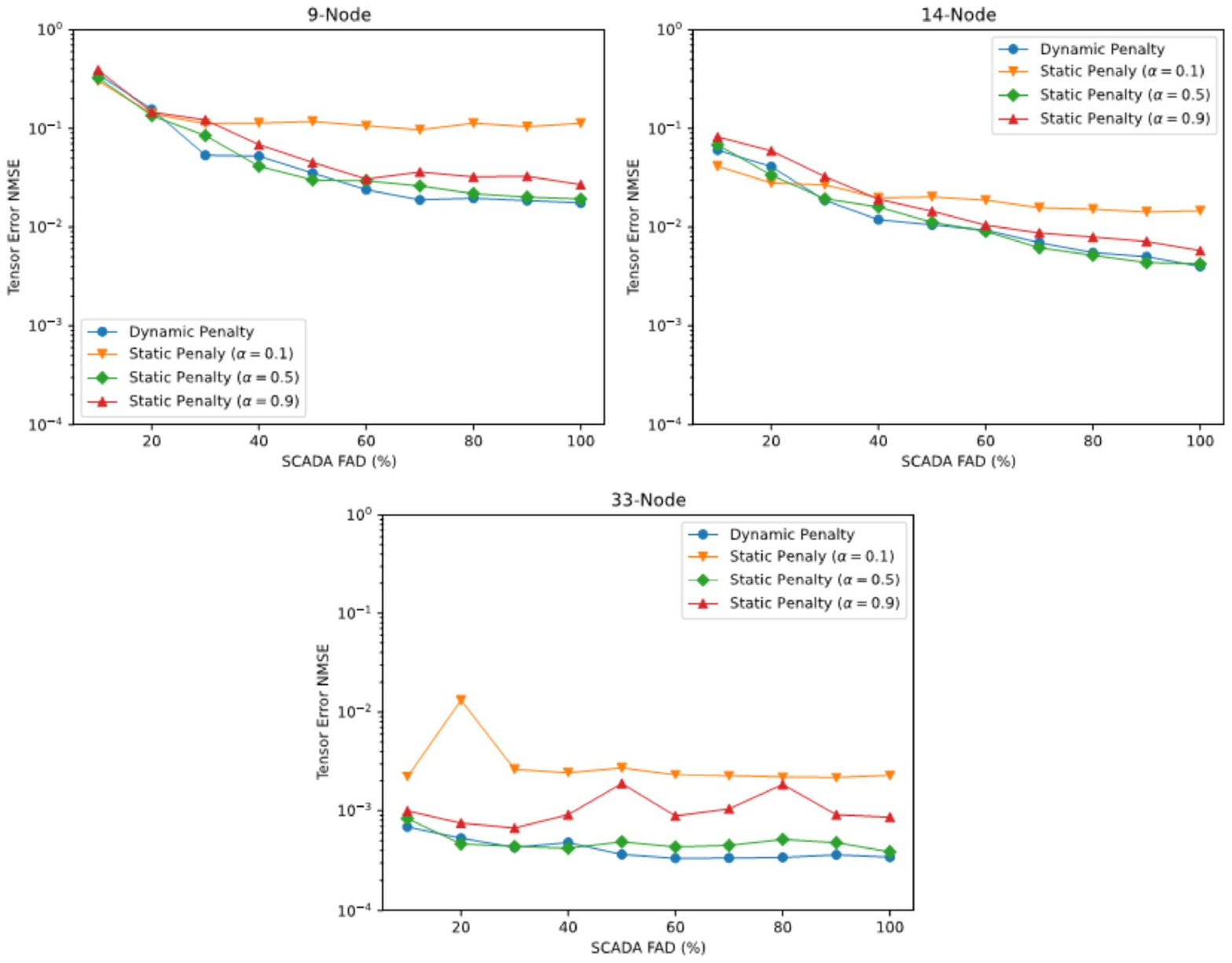}
    \caption{Tensor error of static vs dynamic penalty} 
    \label{fig-static-nmse}
\end{figure}

\begin{figure}
    \includegraphics[width=\textwidth]{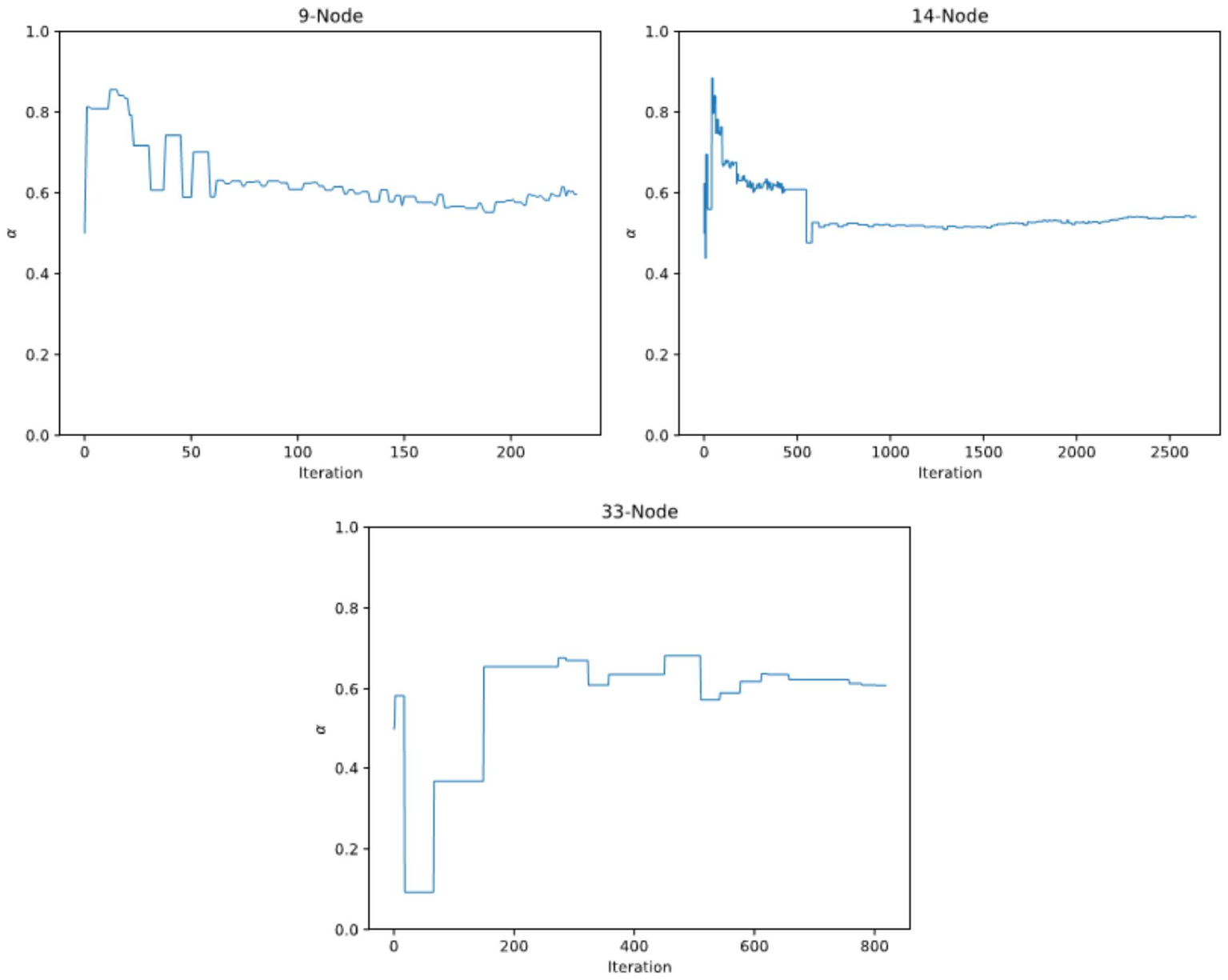}
    \caption{Dynamic Evolution of $\alpha$ at 50\% SCADA FAD.} 
    \label{fig-alpha}
\end{figure}

To assess the effect of the $\alpha$ parameter and of our dynamic penalty scheme, we ran the same experiments with static $\alpha$ values of $0.1,0.5,$ and $0.9$ and report the MAPE of voltage magnitude errors in Figure \ref{fig-static-mag}, the MAE of the voltage angle errors in \ref{fig-static-ang}, and the NMSE of the overall tensor error in Figure \ref{fig-static-nmse}. We see that the choice of $\alpha$ in the static setting can substantially impact estimation error, with $\alpha=0.5$ performing the best with regards to tensor error followed by $\alpha = 0.9$ and $\alpha = 0.1$. In fact, $\alpha = 0.5$ performs very close to the dynamically-derived penalty, as the dynamic penalty in our cases equilibriated between $0.5$ and $0.6$. The evolution of the dynamic $\alpha$ in the 50\% SCADA FAD case is shown in Figure \ref{fig-alpha}. While $\alpha = 0.5$ may have been a fortuitous guess of a reasonable $\alpha$ in our particular tests, we do not have reason to believe this would hold in all settings and networks. For the voltage magnitude, the performance of $\alpha = 0.1$ and $\alpha = 0.9$ is much more sporadic and unpredictable, with $\alpha = 0.9$ performing best in the 33-node network, $\alpha = 0.1$ best in the 14-node, and both generally lagging in the 9-node network. We see again that the dynamic penalty and $\alpha = 0.5$ perform best in the voltage angle errors, and the relative performance of $\alpha=0.1$ and $\alpha =0.9$ varies at different FAD levels. Since $\alpha = 0.1$ applies more weight to the physics penalty and $\alpha =0.9$ applies more to the residuals, this demonstrates that an appropriate $\alpha$ choice is vital for both improved and consistent results. 

The graphs in Figure \ref{fig-alpha} also demonstrate another correlation in the performance of the penalized models. As previously stated, the \emph{Dynamic Penalty} had the lowest voltage angle errors in the 9 and 14-node networks for 30\% SCADA FAD and above. In Figure \ref{fig-alpha}, we see that $\alpha$ initially increases, meaning the residuals have a greater error at the beginning of the algorithm. Meanwhile, $\alpha$ initially decreased in the 33-node network where the \emph{Dynamic Penalty} performed best in voltage magnitude, thus the power flow penalties had a greater initial error. 

\begin{table}
\centering
\caption{Shifted geometric means ($s=10$) of solve times in seconds for each method across all experiments \label{tab:times}}
\begin{tabular}{lccc}
\toprule
                   & 9-node & 14-node & 33-node \\ \midrule
WLS with 100\% FAD & 17.0   & 6.6     & 13.1    \\
BCG into WLS       & 31.5   & 16.4    & 53.1    \\
Static Penalty     & 79.0   & 105.7   & 105.3   \\
Dynamic Penalty    & 86.0   & 121.1   & 100.5   \\ \bottomrule
\end{tabular}
{}
\end{table}

Table \ref{tab:times} shows the shifted geometric means of the solve times for each model in seconds. While the penalized methods have longer times, the results presented involve prototype code; we anticipate that substantial speedups could be obtained from more mature implementation. Moreover, WLS generally runs well within time limits in practice, allowing for perhaps up to an order of magnitude slowdown as an acceptable range.



\section{Conclusion}
We developed a physics-informed tensor completion method for distribution system state estimation with sparse and noisy measurements. The method provides a unified alternative to both purely data-driven tensor completion and classical weighted least squares. Compared with model-free tensor completion, the proposed formulation promotes estimates that are more consistent with network physics. Compared with WLS, it can operate more effectively in low-observability regimes. 

We also derived a deterministic error bound for the estimator. Under a restricted stability condition, the estimation error is controlled by the measurement noise and by the mismatch between the true nonlinear power-flow equations and the linearized physics residuals. This provides an interpretable condition under which data fidelity and physics consistency jointly lead to accurate recovery.

Numerical experiments on standard IEEE networks show that the proposed physics-informed tensor completion models improve state recovery relative to purely data-driven tensor completion and WLS. These results suggest that incorporating convex physics penalties into gauge-norm tensor completion is a promising approach for data-efficient state estimation in distribution networks.

\paragraph{\footnotesize \textbf{Disclaimer}
\noindent The views expressed in this article are those of the authors and do not reflect the official policy or position of the United States Air Force, United States Department of Defense, or United States Government.}

\bibliographystyle{unsrt}  
\bibliography{tse_references}  






\end{document}